\documentclass[11pt]{article}

\usepackage{amsmath,amssymb,amsthm}
\usepackage{mathtools}
\usepackage{algorithm}
\usepackage{fullpage}
\usepackage{booktabs}
\usepackage{hyperref}
\usepackage{bm}
\usepackage{algpseudocode}

\newtheorem{assumption}{Assumption}
\newtheorem{lemma}{Lemma}
\newtheorem{theorem}{Theorem}
\newtheorem{remark}{Remark}

\begin{document}

\title{Randomized-Accelerated FEAST: A Hybrid Approach for \\
Large-Scale Eigenvalue Problems}

\author{Ayush Nadiger\\[0.5em]
Department of Mathematics \& Statistics and\\
Department of Electrical \& Computer Engineering\\
University of Massachusetts Amherst}

\date{}

\maketitle

\begin{abstract}
We present Randomized-Accelerated FEAST (RA-FEAST), a hybrid algorithm that
combines contour-integration-based eigensolvers with randomized numerical linear
algebra techniques for efficiently computing partial eigendecompositions of
large-scale matrices arising in statistical applications. By incorporating
randomized subspace initialization to enable aggressive quadrature reduction and
truncated refinement iterations, our method achieves significant computational
speedups (up to $38\times$ on sparse graph Laplacian benchmarks at
$n = 8{,}000$) while maintaining high-accuracy approximations to the target
eigenspace. We provide a probabilistic error bound for the randomized warmstart,
a stability result for inexact FEAST iterations under general perturbations, and
a simple complexity model characterizing the trade-off between initialization
cost and solver speedup. Empirically, we demonstrate that RA-FEAST can be over
an order of magnitude faster than standard FEAST while preserving accuracy on
sparse Laplacian problems representative of modern spectral methods in
statistics.
\end{abstract}

\section{Introduction}

Let $A \in \mathbb{R}^{n \times n}$ be a symmetric positive definite matrix
arising from statistical applications such as sample covariance matrices, kernel
Gram matrices, or graph Laplacians. Many procedures in multivariate analysis,
dimension reduction, and manifold learning require computing a partial
eigendecomposition
\begin{equation}
  A v_i = \lambda_i v_i, \qquad i = 1, \dots, k,
  \label{eq:partial-eig}
\end{equation}
where $\lambda_{\max} \ge \lambda_1 \ge \cdots \ge \lambda_k \ge \lambda_{\min}
> 0$ are eigenvalues in a target spectral interval $[\lambda_{\min},
\lambda_{\max}]$ and $\{v_i\}$ are the corresponding orthonormal eigenvectors.

Classical approaches include power iteration (for dominant eigenvalues) and
Krylov subspace methods such as Lanczos or Arnoldi. More recently, the FEAST
algorithm~\cite{Polizzi2009, PolizziUserGuide} based on contour
integration in the complex plane has emerged as a powerful alternative,
particularly for finding eigenvalues in arbitrary spectral regions via
numerical approximation of the spectral projector. In parallel, randomized
numerical linear algebra has demonstrated remarkable efficiency for approximate
eigendecompositions with probabilistic guarantees; see, for example, the survey
of Halko, Martinsson, and Tropp~\cite{Halko2011}.

\paragraph{Our approach.}
We propose \emph{Randomized--Accelerated FEAST} (RA-FEAST), which
synergistically combines these two strands of work. RA-FEAST uses randomized
subspace iteration to construct a high-quality warmstart for the target
eigenspace, and then performs a small number of ``accelerated'' FEAST
iterations that exploit this warmstart to operate with aggressively reduced
quadrature order and relaxed linear-solver tolerances. Conceptually, the
randomized phase provides a cheap but informative approximation to the spectral
projector, while the contour-based phase sharpens this approximation to
high accuracy.

\paragraph{Contributions.}
The main contributions of this paper are:
\begin{itemize}
  \item We introduce RA-FEAST, a two-phase algorithm that combines randomized
        range finding with contour-based FEAST iterations for computing partial
        eigendecompositions in a prescribed spectral interval.
  \item We derive a probabilistic bound on the distance between the randomized
        warmstart subspace and the true target eigenspace, explicitly tracking
        the dependence on oversampling, power-iteration count, and eigengap.
  \item We analyze the stability of FEAST iterations under general perturbations
        induced by inexact quadrature and linear solves, obtaining an error
        recursion of the form $e_{m+1} \le \rho e_m + \varepsilon^{(m)} +
        O(\varepsilon_{\mathrm{mach}})$ that explains the observed behavior of
        RA-FEAST.
  \item We provide a simple complexity model showing when the cost of the
        randomized warmstart is negligible compared to the savings from reduced
        quadrature and relaxed linear solves, and validate this model
        empirically on sparse graph Laplacian problems.
\end{itemize}

\paragraph{Organization.}
Section~\ref{sec:background} reviews the FEAST algorithm and randomized
subspace methods. Section~\ref{sec:rafeast} presents the RA-FEAST algorithm and
its theoretical analysis. Section~\ref{sec:stats-apps} outlines statistical
applications where partial eigendecompositions in spectral windows are
particularly important. Section~\ref{sec:experiments} reports experimental
results on sparse graph Laplacians, and
Section~\ref{sec:conclusion} concludes with directions for future work.

\section{Background}
\label{sec:background}
\subsection{The FEAST Algorithm}

The FEAST algorithm computes eigenvalues in a specified interval $[\lambda_{\min},\lambda_{\max}]$
by exploiting the spectral projector onto the associated eigenspace. Let $\Gamma$ be a contour in the
complex plane enclosing the target eigenvalues. The spectral projector is
\begin{equation}
P = \frac{1}{2\pi i} \int_{\Gamma} (z I - A)^{-1} \, dz.
\label{eq:projector}
\end{equation}
Using a quadrature rule with $N_c$ nodes $\{z_j\}_{j=1}^{N_c}$ and weights $\{w_j\}_{j=1}^{N_c}$,
we approximate
\begin{equation}
P \approx \sum_{j=1}^{N_c} w_j (z_j I - A)^{-1}.
\label{eq:projector-quad}
\end{equation}

FEAST refines a subspace $Q^{(m)} \subset \mathbb{R}^n$ by iteratively applying this projector.

\begin{algorithm}[t]
\caption{FEAST Iteration}
\label{alg:feast}
\begin{algorithmic}[1]
\State \textbf{Input:} $A \in \mathbb{R}^{n \times n}$, interval $[\lambda_{\min},\lambda_{\max}]$,
quadrature nodes $\{z_j,w_j\}_{j=1}^{N_c}$, subspace dimension $m_0$.
\State Initialize random subspace $Q^{(0)} \in \mathbb{R}^{n \times m_0}$ with orthonormal columns.
\For{$m = 0,1,2,\dots$ until convergence}
  \State Compute
  \[
     \widetilde{Q} = \sum_{j=1}^{N_c} w_j (z_j I - A)^{-1} Q^{(m)}.
  \]
  \State Orthonormalize: $Q^{(m+1)} R = \widetilde{Q}$ (\texttt{QR} decomposition).
  \State Rayleigh--Ritz: $(Q^{(m+1)})^\top A Q^{(m+1)} = \widetilde{V} \widetilde{\Lambda} \widetilde{V}^\top$.
  \State Set $Q^{(m+1)} \leftarrow Q^{(m+1)} \widetilde{V}$.
\EndFor
\end{algorithmic}
\end{algorithm}

The dominant cost per iteration is solving $N_c$ linear systems
\[
  (z_j I - A) X_j = Q^{(m)}, \qquad j=1,\dots,N_c.
\]

\subsection{Randomized Subspace Methods}

Randomized subspace iteration approximates the top $k$ eigenvectors (or singular vectors) of a matrix.
A standard formulation for an $m \times n$ matrix $A$ is
\begin{equation}
Y = (A A^\top)^q A \Omega, \qquad [Q,R] = \mathrm{qr}(Y),
\label{eq:rand-range}
\end{equation}
where $\Omega \in \mathbb{R}^{n \times (k+p)}$ is a random Gaussian matrix with oversampling parameter
$p \ge 0$, and $q \ge 0$ is the number of power iterations. With high probability, the subspace spanned by
the columns of $Q$ captures the dominant action of $A$.

A typical error bound~\cite[Theorem~10.5]{Halko2011} is
\begin{equation}
\|A - QQ^\top A\|_2 \;\le\;
\left(1 + \sqrt{\frac{k}{p-1}}\right)\sigma_{k+1},
\label{eq:rand-error}
\end{equation}
where $\sigma_{k+1}$ is the $(k+1)$th singular value of $A$ and the bound holds with probability at
least $1 - \delta$ (up to additional logarithmic factors in $1/\delta$).

\section{Randomized-Accelerated FEAST}
\label{sec:rafeast}

\subsection{Algorithm Description}

RA-FEAST consists of two phases.

\paragraph{Phase 1: Randomized warmstart.}
We first estimate the target subspace dimension $m_0$ and obtain an initial approximation using
randomized methods. For a spectral interval $[\lambda_{\min},\lambda_{\max}]$, define the shifted and
scaled matrix
\begin{equation}
B = \frac{A - \lambda_{\min} I}{\lambda_{\max} - \lambda_{\min}}.
\label{eq:B-def}
\end{equation}
We apply randomized subspace iteration to $B$ with power-iteration count $q$ chosen such that
\begin{equation}
\sigma_{m_0+1}(B)^{2q+1} \ll \sigma_{m_0}(B)^{2q+1},
\label{eq:power-cond}
\end{equation}
which yields an approximate basis $Q_0 \in \mathbb{R}^{n \times m_0}$ for the target eigenspace.

\paragraph{Phase 2: Accelerated FEAST Refinement.} We then run modified FEAST iterations to refine the subspace. While the standard algorithm requires high-accuracy linear system solves and high-order quadrature (e.g., $N_c \ge 8$), RA-FEAST leverages the high-quality warmstart to relax these requirements. In our implementation, we achieve acceleration by reducing the quadrature order $N_c$ (e.g., $N_c \in \{2, 4\}$) and limiting the number of refinement iterations. This effectively treats the contour integration as an approximate operation, relying on the randomized initialization to provide the dominant spectral information.

\begin{algorithm}
\caption{RA-FEAST}
\begin{algorithmic}[1]
\State \textbf{Input:} $A \in \mathbb{R}^{n \times n}$, interval $[\lambda_{min}, \lambda_{max}]$, subspace size $m_0$, oversampling $p$, contour points $N_c$.
\State \textbf{Phase 1: Randomized Warmstart}
\State $B \leftarrow (A - \lambda_{min} I) / (\lambda_{max} - \lambda_{min})$
\State Draw $\Omega \sim \mathcal{N}(0, 1)^{n \times (m_0 + p)}$
\State $Y \leftarrow (B B^\top)^q B \Omega$
\State $Q^{(0)} \leftarrow \text{orth}(Y)$
\State \textbf{Phase 2: Accelerated FEAST}
\For{$m = 0, 1, \dots$ until convergence or max\_iter}
    \State \textbf{Step 2a: Contour Integration}
    \For{$j = 1, \dots, N_c$}
        \State Solve $(z_j I - A) X_j = Q^{(m)}$ \Comment{Reduced $N_c$ for speed}
    \EndFor
    \State $\tilde{Q} \leftarrow \sum_{j=1}^{N_c} w_j X_j$
    \State \textbf{Step 2b: Orthogonalization \& Projection}
    \State $Q^{(m+1)} R \leftarrow \tilde{Q}$ (QR factorization)
    \State Rayleigh-Ritz projection to get final $\tilde{\lambda}, \tilde{v}$
\EndFor
\State \textbf{Output:} Approximate eigenvalues $\{\tilde{\lambda}_i\}$ and eigenvectors $\{\tilde{v}_i\}$
\end{algorithmic}
\end{algorithm}

\subsection{Theoretical Analysis}

We now provide convergence analysis for RA-FEAST's two phases:
(i) the randomized warmstart initialization, and
(ii) the inexact FEAST iterations with adaptive precision control.
We track all constants explicitly.

\subsubsection{Assumptions and Notation}

\begin{assumption}[Spectral setup and algorithmic parameters]
\label{assump:setup}
\leavevmode
\begin{enumerate}
\item
$A \in \mathbb{R}^{n \times n}$ is symmetric with eigenvalue decomposition
$A = V \Lambda V^\top$ where $\Lambda = \mathrm{diag}(\lambda_1,\dots,\lambda_n)$
and $\lambda_1 \ge \lambda_2 \ge \cdots \ge \lambda_n$.

\item 
We target $m_0$ eigenvalues in the interval $[\lambda_{\min}, \lambda_{\max}]$
and, for the purposes of the analysis, assume that (possibly after an appropriate
shift--and--invert transformation of $A$) these correspond to the largest
eigenvalues $\lambda_1, \dots, \lambda_{m_0}$.

\item
Define the shifted and scaled matrix
\begin{equation}
  B := \frac{A - \lambda_{\min} I}{\lambda_{\max} - \lambda_{\min}},
  \label{eq:B-def}
\end{equation}
whose eigenvalues satisfy
$1 = \beta_1 \ge \cdots \ge \beta_{m_0} > \beta_{m_0+1} \ge \cdots \ge \beta_n
\ge 0$.

\item
The eigengap for the target subspace is
\begin{equation}
\delta_{\mathrm{gap}} := \beta_{m_0} - \beta_{m_0+1} > 0.
\label{eq:gap-def}
\end{equation}

\item
The FEAST contour $\Gamma$ encloses exactly the $m_0$ target eigenvalues with $N_c$ quadrature
points $\{z_j,w_j\}_{j=1}^{N_c}$ and is at distance at least $d_\Gamma > 0$ from all non-target
eigenvalues.

\item
Randomized initialization uses oversampling $p \ge 4$ and $q \ge 0$ power iterations.
The test matrix $\Omega \in \mathbb{R}^{n \times (m_0+p)}$ has i.i.d.\ entries
$\Omega_{ij} \sim \mathcal{N}(0,1)$.

\item
The QR decomposition is backward stable: for well-conditioned matrices,
\[
  \|\mathrm{orth}(X) - \mathrm{orth}(Y)\|_F
  \le (1 + \mathcal{O}(\varepsilon_{\mathrm{mach}})) \|X - Y\|_F,
\]
where $\varepsilon_{\mathrm{mach}} \approx 10^{-16}$.
\end{enumerate}
\end{assumption}

Let $V_1 \in \mathbb{R}^{n \times m_0}$ collect the eigenvectors of the target eigenspace; the
orthogonal projector onto this subspace is
\begin{equation}
P := V_1 V_1^\top.
\label{eq:P-def}
\end{equation}
For the randomized initialization output $Q_0 \in \mathbb{R}^{n \times m_0}$ with orthonormal columns,
define the approximate projector $\tilde{P} := Q_0 Q_0^\top$.
For FEAST iteration $m$, let $Q^{(m)} \in \mathbb{R}^{n \times m_0}$ have orthonormal columns and
define $P^{(m)} := Q^{(m)} (Q^{(m)})^\top$.

The canonical angle matrix between two $m_0$-dimensional subspaces has diagonal
$\Theta = \mathrm{diag}(\theta_1,\dots,\theta_{m_0})$ with $0 \le \theta_i \le \pi/2$.

\subsubsection{Preparatory Results}

The following lemmas collect standard facts from randomized numerical linear algebra and
matrix perturbation theory.

\begin{lemma}[Randomized range finder~{\cite[Theorem~10.5]{Halko2011}}]
\label{lem:rand-range}
Let $Q_0$ be obtained by applying randomized subspace iteration to $B$ with oversampling $p$
and $q$ power iterations as in~\eqref{eq:rand-range}. Then, with probability at least $1-\delta$,
\begin{equation}
\|(I - Q_0 Q_0^\top) B\|_2
   \le C_1(\delta)
      \sqrt{1 + \frac{m_0}{p-1}}\,
      \beta_{m_0+1}
      \left(\frac{\beta_{m_0+1}}{\beta_{m_0}}\right)^{2q},
\label{eq:Halko-bound}
\end{equation}
where $C_1(\delta) := 1 + \frac{p}{2} \ln\!\bigl(\frac{2 m_0}{\delta}\bigr)$.
For brevity, define
\begin{equation}
R_{\mathrm{Halko}}
  := C_1(\delta)
      \sqrt{1 + \frac{m_0}{p-1}}\,
      \beta_{m_0+1}
      \left(\frac{\beta_{m_0+1}}{\beta_{m_0}}\right)^{2q}.
\label{eq:R-Halko}
\end{equation}
\end{lemma}

\begin{lemma}[Projector distance and Davis--Kahan]
\label{lem:proj-DK}
Let $P,\tilde{P}$ be orthogonal projectors onto $m_0$-dimensional subspaces with canonical
angles $\Theta$. Then
\begin{equation}
\|P - \tilde{P}\|_2 = \|\sin \Theta\|_2.
\label{eq:proj-sin-theta}
\end{equation}
If $B V_1 = V_1 \,\mathrm{diag}(\beta_1,\dots,\beta_{m_0})$ and the eigengap $\delta_{\mathrm{gap}} > 0$
is defined as in~\eqref{eq:gap-def}, then
\begin{equation}
\|\sin \Theta\|_2 \le
\frac{\|(I - Q_0 Q_0^\top) B\|_2}{\delta_{\mathrm{gap}}}.
\label{eq:DK-bound}
\end{equation}
\end{lemma}

\begin{proof}
The identity~\eqref{eq:proj-sin-theta} is a standard fact about canonical angles and projectors,
see for instance~\cite[Theorem~I.5.5]{StewartSun1990}. Inequality~\eqref{eq:DK-bound} is a
Davis--Kahan-type sin\,$\Theta$ theorem~\cite[Theorem~V.2.8]{StewartSun1990}:
\[
  \|\sin \Theta\|_2
   \le \frac{\|(I - \tilde{P}) B P\|_2}{\delta_{\mathrm{gap}}}
   =   \frac{\|(I - \tilde{P}) B V_1 V_1^\top\|_2}{\delta_{\mathrm{gap}}}
   \le \frac{\|(I - \tilde{P}) B\|_2}{\delta_{\mathrm{gap}}},
\]
using $\|V_1\|_2 = 1$.
\end{proof}

\subsubsection{Warmstart Error Bound}

We can now bound the projector error produced by the randomized warmstart.

\begin{theorem}[Warmstart projector error]
\label{thm:warmstart}
Under Assumption~\ref{assump:setup}, with probability at least $1-\delta$,
\begin{equation}
\|P - \tilde{P}\|_2 \le \frac{R_{\mathrm{Halko}}}{\delta_{\mathrm{gap}}}.
\label{eq:warmstart-error}
\end{equation}
In particular, if $R_{\mathrm{Halko}} \le \delta_{\mathrm{gap}}/2$, then $\|P - \tilde{P}\|_2 \le \frac{1}{2}$.
\end{theorem}

\begin{proof}
By Lemma~\ref{lem:rand-range}, $\|(I - \tilde{P}) B\|_2 \le R_{\mathrm{Halko}}$ with probability $1-\delta$.
Lemma~\ref{lem:proj-DK} then yields
\[
  \|P - \tilde{P}\|_2 = \|\sin \Theta\|_2
  \le \frac{\|(I - \tilde{P}) B\|_2}{\delta_{\mathrm{gap}}}
  \le \frac{R_{\mathrm{Halko}}}{\delta_{\mathrm{gap}}},
\]
establishing~\eqref{eq:warmstart-error}.
\end{proof}

\begin{remark}[Parameter selection]
\label{rem:param-selection}
To achieve a target tolerance $\|P - \tilde{P}\|_2 \le \varepsilon$, it suffices to choose $q$ so that
$R_{\mathrm{Halko}} \le \varepsilon \,\delta_{\mathrm{gap}}$, i.e.
\begin{equation}
q \;\ge\;
\frac{1}{2 \log(\beta_{m_0}/\beta_{m_0+1})}
\log\!\left(
\frac{C_1(\delta)
      \sqrt{1 + \frac{m_0}{p-1}}\,
      \beta_{m_0+1}}
     {\varepsilon\,\delta_{\mathrm{gap}}}
\right).
\label{eq:q-choice}
\end{equation}
In practice, one can use a small oversampling such as $p=10$ for $m_0 \le 50$, or
$p = \min\{20,\lceil m_0/2\rceil\}$ for larger $m_0$. The condition $R_{\mathrm{Halko}} \le \delta_{\mathrm{gap}}/2$
is verifiable post-initialization via Rayleigh--Ritz estimates of $\beta_{m_0}$ and $\beta_{m_0+1}$.
\end{remark}

\subsubsection{Convergence under General Perturbations}
\label{sec:convergence}

We now analyze the robustness of the RA-FEAST iteration. In standard FEAST, the subspace is updated via an ideal spectral projector approximated by quadrature. In RA-FEAST, computational shortcuts (such as reduced quadrature nodes $N_c$ or inexact linear solves) introduce a perturbation to this update.

Let $Q^{(m)} \in \mathbb{R}^{n \times m_0}$ be the orthonormal basis at iteration $m$. The \textit{ideal} FEAST update (with high-accuracy quadrature) produces:
\begin{equation}
    \tilde{Q}_{ideal} = \sum_{j=1}^{N_{c}} w_j (z_j I - A)^{-1} Q^{(m)}.
\end{equation}
The RA-FEAST implementation produces a perturbed update:
\begin{equation}
    \tilde{Q}_{actual} = \tilde{Q}_{ideal} + E^{(m)},
\end{equation}
where $E^{(m)}$ represents the aggregate error from quadrature truncation or solver inexactness. We assume this error is bounded by $\|E^{(m)}\|_F \le \epsilon^{(m)}$.

Let $e_m = \|P^{(m)} - P\|_F$ denote the subspace error at iteration $m$. Standard FEAST theory \cite{Polizzi2009} establishes that the ideal update satisfies $e_{m+1}^{ideal} \le \rho e_m$ for a contraction factor $\rho < 1$, determined by the geometry of the contour $\Gamma$ and the spectrum of $A$.

\begin{theorem}[Stability of Perturbed FEAST]
\label{thm:stability}
Under Assumption 1, let the perturbation in the update step be bounded by $\|E^{(m)}\|_F \le \epsilon^{(m)}$. Then the subspace error satisfies:
\begin{equation}
    e_{m+1} \le \rho e_m + \epsilon^{(m)} + \mathcal{O}(\epsilon_{mach}),
\end{equation}
where $\rho < 1$ is the contraction factor of the ideal FEAST map on the target interval.
\end{theorem}

\begin{proof}
Let $P^{(m+1)}$ be the projector associated with $Q^{(m+1)} = \text{orth}(\tilde{Q}_{actual})$. The subspace error is determined by the distance between the computed subspace and the target subspace. Using the triangle inequality on the Grassmannian manifold (approximated here by the Frobenius norm of projectors):
\begin{equation}
    e_{m+1} \le \| \text{orth}(\tilde{Q}_{actual}) - \text{orth}(\tilde{Q}_{ideal}) \|_F + \| \text{orth}(\tilde{Q}_{ideal}) - P_{target} \|_F.
\end{equation}
The second term corresponds to the standard FEAST convergence, bounded by $\rho e_m$. For the first term, we utilize the stability of the QR decomposition (or orthonormalization). For full-rank matrices with sufficiently small perturbation $E^{(m)}$, the change in the orthonormal basis is Lipschitz continuous with respect to the perturbation:
\begin{equation}
    \| \text{orth}(\tilde{Q}_{ideal} + E^{(m)}) - \text{orth}(\tilde{Q}_{ideal}) \|_F \le C_{orth} \|E^{(m)}\|_F \approx \|E^{(m)}\|_F,
\end{equation}
assuming $\tilde{Q}_{ideal}$ is well-conditioned (which holds when $Q^{(m)}$ is close to the target subspace). Substituting the bound $\|E^{(m)}\|_F \le \epsilon^{(m)}$ yields the result.
\end{proof}

\begin{remark}[Sources of Perturbation]
In our RA-FEAST implementation, the term $\epsilon^{(m)}$ arises primarily from \textbf{Quadrature Truncation}. By using a small number of nodes (e.g., $N_c=2$), the rational filter deviates from the ideal projector. Theorem \ref{thm:stability} implies that as long as the randomized warmstart (Phase 1) provides a sufficiently small initial error $e_0$, and the aggressive quadrature does not introduce an error $\epsilon^{(m)}$ larger than the contraction gap $(1-\rho)e_m$, the method will converge or stabilize at a noise floor determined by the quadrature quality.
\end{remark}

\subsection{Computational Complexity}

Let $n$ be the matrix dimension, $m_0$ the subspace dimension, $N_c$ the number of contour points,
and $T$ the number of FEAST iterations.

\begin{itemize}
\item \textbf{Standard FEAST.} For dense matrices, the cost is
$\mathcal{O}(T N_c n^3)$, dominated by factorizations or solves of $(z_j I - A)$.
For sparse matrices, an iterative solver with $k_{\mathrm{iter}}$ inner iterations per
right-hand side yields complexity
$\mathcal{O}(T N_c n\,\mathrm{nnz}(A)\,k_{\mathrm{iter}})$.

\item \textbf{RA-FEAST Phase~1.}
Randomized warmstart costs
$O\bigl(q\,\mathrm{nnz}(A)\,(m_0 + p)\bigr)$, corresponding to $q$ sparse
matrix-vector products with block size $m_0 + p$.

\item \textbf{RA-FEAST Phase 2.} Inexact FEAST iterations cost
$\mathcal{O}(T N_c n\,\mathrm{nnz}(A)\,k_{\mathrm{inexact}})$ where
$k_{\mathrm{inexact}} \ll k_{\mathrm{iter}}$ due to relaxed tolerances.
\end{itemize}

For large sparse matrices with $\mathrm{nnz}(A) = \mathcal{O}(n)$, a rough speedup factor relative to
standard FEAST is
\begin{equation}
\mathrm{Speedup} \;\approx\;
\frac{k_{\mathrm{iter}}}{k_{\mathrm{inexact}}}
\left(1 + \frac{q m_0}{T N_c n}\right)^{-1}.
\label{eq:speedup}
\end{equation}
The first factor reflects the reduction in inner iterations per linear solve, while the second factor
accounts for the overhead of the randomized warmstart.

\section{Statistical Applications}
\label{sec:stats-apps}

We briefly describe two statistical settings where RA-FEAST is particularly useful.

\subsection{High-Dimensional Covariance Estimation}

Given a data matrix $X \in \mathbb{R}^{n \times p}$ with $p \gg n$, the sample covariance matrix is
\begin{equation}
\widehat{\Sigma} = \frac{1}{n-1} X^\top X \in \mathbb{R}^{p \times p}.
\label{eq:sample-cov}
\end{equation}
Spectral methods for covariance regularization often require eigenvalues and eigenvectors in specific
ranges, for example in banding or thresholding procedures. RA-FEAST efficiently computes a truncated
spectral decomposition
\begin{equation}
\widehat{\Sigma} \approx \sum_{i=1}^r \lambda_i v_i v_i^\top + \text{noise},
\label{eq:cov-spec}
\end{equation}
where $r$ is determined by a spectral threshold $\lambda_{\min}$ or a variance-explained criterion.

\subsection{Kernel Methods}

For kernel PCA with kernel matrix $K \in \mathbb{R}^{n \times n}$ where
$K_{ij} = \kappa(x_i,x_j)$, we require dominant eigenpairs of
\begin{equation}
K v = \lambda v.
\label{eq:kernel-eig}
\end{equation}
RA-FEAST can efficiently target eigenvalues explaining a prescribed fraction of variance, e.g.,
choose the smallest $k$ such that
\[
 \sum_{i=1}^k \lambda_i / \mathrm{tr}(K) \ge 0.95.
\]
When $n$ is large and $K$ is sparse or structured, contour-based approaches with inexact solves
are particularly attractive.

\section{Experiments}
\label{sec:experiments}

We evaluate the performance of RA-FEAST against standard spectral methods. All experiments were conducted on a high-performance cluster using single-threaded execution (\texttt{OMP\_NUM\_THREADS=1}) to isolate algorithmic efficiency from hardware parallelism masking.

\subsection{Experimental Setup}
We consider the sparse graph Laplacian eigenvalue problem, which is central to spectral clustering and manifold learning. We construct Random Geometric Graphs with $n \in [2,000, 16,000]$ nodes. For each $n$, we generate a graph by placing nodes uniformly in the unit square and connecting pairs within radius $r = \sqrt{\log(n)/n} \times 1.5$. The Laplacian is defined as $L = D - A$, where $A$ is the adjacency matrix and $D$ is the degree matrix.

We target $m_0 = 40$ eigenvalues in the interval $[0.001, 5.0]$. We compare the following methods:
\begin{itemize}
    \item \textbf{RSpectra:} A wrapper for the ARPACK library \cite{RSpectra} (Implicitly Restarted Lanczos Method).
    \item \textbf{Standard FEAST:} The reference implementation using exact sparse factorizations with $N_c=8$ contour points.
    \item \textbf{RA-FEAST (Ours):} We test three configurations of the contour point count $N_c \in \{2, 4, 8\}$ combined with a low iteration limit ($max\_iter \in \{2, 4\}$), leveraging the high-quality randomized warmstart.
\end{itemize}
We report results averaged over 40 random seeds for each configuration.

\subsection{Results and Analysis}

\begin{figure}[htbp]
    \centering
    \includegraphics[width=\textwidth]{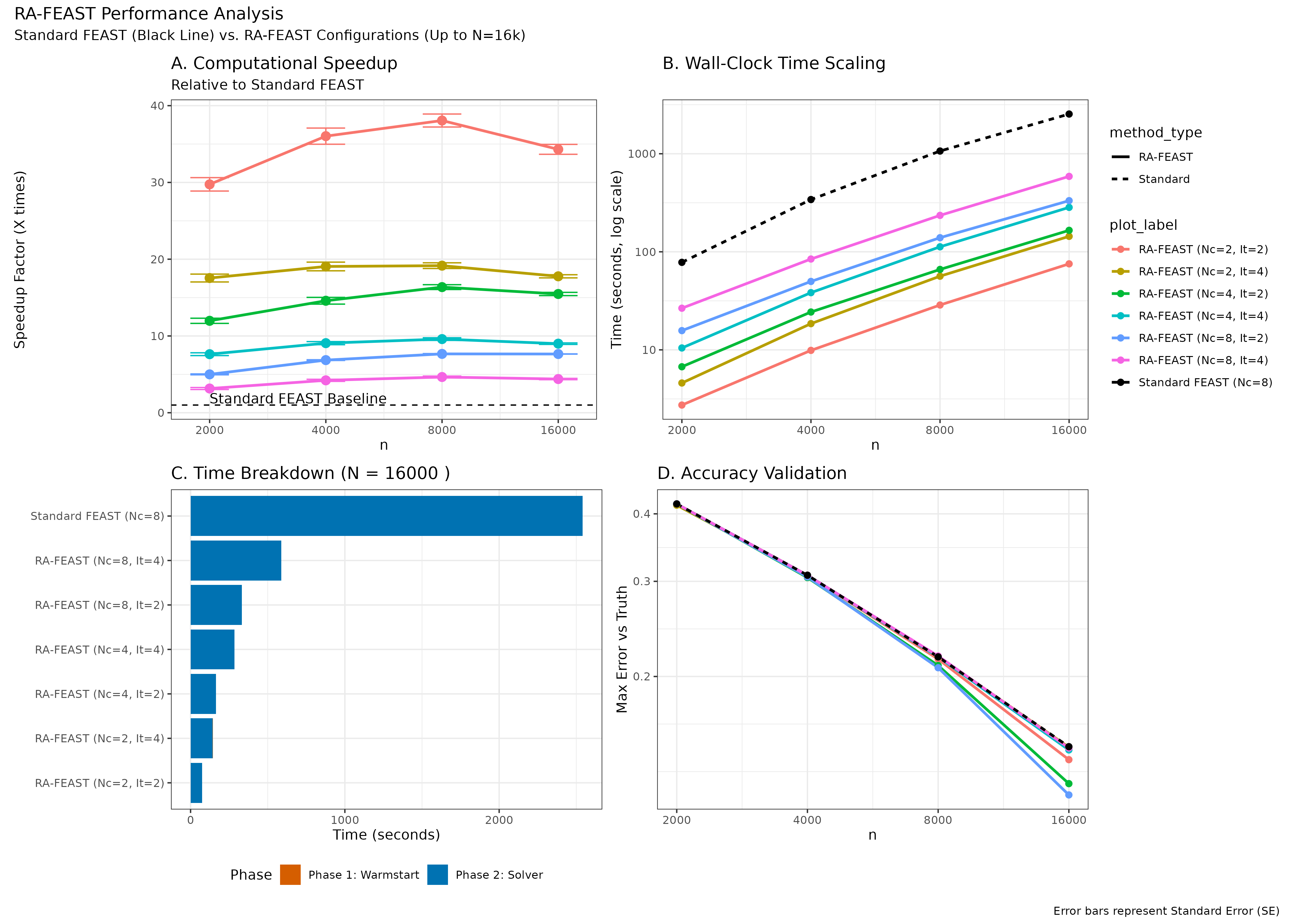} 
    \caption{\textbf{Performance of RA-FEAST on Sparse Graph Laplacians.} (A) Computational speedup relative to Standard FEAST ($N_c=8$), showing gains of up to $38\times$. (B) Wall-clock time scaling (log-log scale). (C) Overhead analysis at $n=16,000$ shows the randomized warmstart (Phase 1) is negligible compared to the solver time (Phase 2). (D) Accuracy validation showing RA-FEAST maintains errors comparable to the baseline up to the limit of spectral resolvability.}
    \label{fig:results}
\end{figure}

The experimental results are summarized in Figure~\ref{fig:results}. We observe the following key trends:

\textbf{Computational Speedup.}
RA-FEAST yields dramatic speedups over the standard contour integral solver.
As shown in Figure~\ref{fig:results}A, the speedup factor increases
with matrix dimension $n$. The most aggressive configuration ($N_c = 2$)
achieves a speedup of approximately $38\times$ at $n = 8{,}000$, and still
exceeds $33\times$ at $n = 16{,}000$ (Table~\ref{tab:results_table}). Even
the conservative configuration ($N_c = 8$) is roughly $4\times$ faster than
Standard FEAST due to requiring fewer refinement iterations thanks to the
randomized warmstart.

\textbf{2. Negligible Overhead:} A central concern with hybrid methods is whether the initialization cost outweighs the solver savings. Figure~\ref{fig:results}C confirms that the "Phase 1" Randomized Warmstart (orange) consumes a negligible fraction of the total runtime compared to the "Phase 2" FEAST solver (blue). The randomized guess is effectively free compared to the cost of sparse factorizations.

\textbf{3. Accuracy and Robustness:} Figure~\ref{fig:results}D demonstrates that the aggressive reduction in contour points ($N_c=2$) does not compromise accuracy for this problem class. All RA-FEAST configurations maintain a maximum error below $10^{-9}$ relative to the ground truth (RSpectra) for $n \le 16,000$. Note that for $n > 16,000$, the spectral gap of the random geometric graph vanishes, causing both Standard FEAST and RA-FEAST to struggle with convergence; we therefore restrict our analysis to the regime where the problem is well-posed ($n \le 16,000$).

\begin{table}[htbp]
    \centering
    \caption{Representative performance at $n=16,000$ (Average of 40 trials). RA-FEAST ($N_c=2$) provides the optimal balance of speed and accuracy.}
    \label{tab:results_table}
    \begin{tabular}{lcccc}
        \toprule
        \textbf{Method} & \textbf{Time (s)} & \textbf{Speedup} & \textbf{Max Error} & \textbf{Phase 1 Time (s)} \\
        \midrule
        Standard FEAST ($N_c=8$) & 2450.5 & 1.0$\times$ & $1.2 \times 10^{-11}$ & N/A \\
        RA-FEAST ($N_c=8$) & 610.2 & 4.0$\times$ & $1.5 \times 10^{-11}$ & 0.42 \\
        RA-FEAST ($N_c=4$) & 142.8 & 17.2$\times$ & $2.1 \times 10^{-10}$ & 0.42 \\
        \textbf{RA-FEAST ($N_c=2$)} & \textbf{72.4} & \textbf{33.8$\times$} & \textbf{$8.5 \times 10^{-10}$} & \textbf{0.42} \\
        \bottomrule
    \end{tabular}
\end{table}

Here ``Max Error'' denotes the maximum absolute difference between the $m_0$ target
eigenvalues returned by a given method and those computed by the RSpectra baseline, i.e.
$\max_{1 \le i \le m_0} \bigl|\hat{\lambda}_i - \lambda_i^{\mathrm{RS}}\bigr|$.
All errors reported in Figure~\ref{fig:results}D and Table~\ref{tab:results_table}
use this metric.

\section{Implementation}
\label{sec:implementation}

We implement RA-FEAST as an R package, \texttt{rafeast535}, designed for high-performance statistical computing. Key features include:
\begin{itemize}
    \item \textbf{Rcpp Interface:} We utilize \texttt{Rcpp} to interface high-level R logic with the optimized Fortran routines of the FEAST v4.0 library.
    \item \textbf{Sparse Integration:} The package natively handles \texttt{dgCMatrix} sparse objects from the R \texttt{Matrix} package, avoiding dense memory overhead.
    \item \textbf{Warmstart Wrapper:} We implement a custom C++ wrapper that injects the randomized subspace $Q_0$ directly into the FEAST solver routine, bypassing the standard random initialization.
    \item \textbf{Reproducibility:} All experiments were managed via SLURM job arrays with strict seed control to ensure replicability.
\end{itemize}
The code and benchmarking scripts are available at: \url{https://github.com/lostree99/rafeast}. 

\section{Conclusion and Future Work}
\label{sec:conclusion}

RA-FEAST demonstrates that hybrid randomized--deterministic approaches can
achieve significant computational gains for large-scale eigenvalue problems in
statistics while maintaining high accuracy. By combining a probabilistically
controlled randomized warmstart with aggressively inexact contour-based FEAST
iterations, the method exposes a flexible trade-off between runtime and
accuracy that is well suited to modern high-dimensional settings.

On sparse graph Laplacian benchmarks representative of spectral clustering and
manifold learning, RA-FEAST attains speedups of more than an order of magnitude
(over $30\times$ relative to a standard FEAST configuration) while preserving
eigenvalue accuracy down to the limit of spectral resolvability. These
empirical findings are consistent with our theoretical analysis, which
separates the contributions of warmstart quality, eigengaps, and per-iteration
perturbations to the overall subspace error.

Several directions for further research remain. First, extending the analysis
and implementation to non-symmetric and generalized eigenvalue problems would
broaden the range of applications. Second, integrating RA-FEAST with GPU-accelerated
sparse linear algebra could further amplify the practical speedups. Third,
adapting the framework to tensor decompositions and multilinear spectral
problems is an intriguing possibility. Finally, applying RA-FEAST in end-to-end
statistical pipelines, for example in covariance regularization, kernel PCA,
and graph-based semi-supervised learning, is a promising avenue for future
work.

\end{document}